\documentclass[reqno,12pt]{amsart}

\usepackage{amsmath, epsfig, cite}
\usepackage{amssymb}
\usepackage{amsfonts}
\usepackage{latexsym}
\usepackage{amsthm}

\usepackage{amsmath, epsfig, cite}
\usepackage{amssymb}
\usepackage{amsfonts}
\usepackage{latexsym}
\usepackage{amsthm}
\newtheorem{theorem}{Theorem}[section]

\newtheorem{corollary}[theorem]{Corollary}
\newtheorem{lemma}[theorem]{Lemma}

\theoremstyle{remark}

\numberwithin{equation}{section}

\newcommand{\N}{\mathbb N}
\newcommand{\Z}{\mathbb Z}

\allowdisplaybreaks

\author{Xiaoxia Wang }
\address{DEPARTMENT OF MATHEMATICS, SHANGHAI UNIVERSITY, SHANGHAI 200444, P. R. CHINA}
\email{$^*$ Corresponding author. xiaoxiawang@shu.edu.cn (X. Wang), xchangi@shu.edu.cn (C. Xu).}
\thanks{This work is supported by Natural Science Foundations of Shanghai (22ZR1424100).}
\author{Chang Xu$^*$}

\title{$q$-Supercongruences on triple and quadruple sums }



\begin{document}
\begin{abstract}
Inspired by the recent work of El Bachraoui, we present some
new $q$-supercongruences on triple and quadruple  sums of basic hypergeometric series.
In particular, we give a $q$-supercongruence modulo the fifth power of a
cyclotomic polynomial, which is a $q$-analogue of the quadruple sum  of
Van Hamme's supercongruence (G.2).
\end{abstract}
\maketitle
\section{Introduction}
In 1913, Ramanujan (see \cite{Berndt}) proposed some well-known hypergeometric series without any proof, including
\begin{equation*}
\sum_{k=0}^\infty(8k+1)\frac{\left(\frac{1}{4}\right)^4_k}{k!^4}
=\frac{2\sqrt{2}}{\sqrt{\pi}\Gamma\left(\frac{3}{4}\right)^2}.
\end{equation*}
Here $(a)_n=a(a+1)\cdots(a+n-1)$ is the Pochhammer symbol and $\Gamma(x)$ is the gamma function.
In 1997, Van Hamme \cite{Van} studied partial sums of Ramanujan's summation formular for $1/\pi$
 and conjectured $13$ Ramanujan-type supercongruences, involving
\begin{align*}
&\mathrm{(G.2)}\quad\sum_{k=0}^{\frac{p-1}{4}}(8k+1)\frac{\left(\frac{1}{4}\right)^4_k}{k!^4}
\equiv\frac{p\Gamma_p\left(\frac{1}{4}\right)\Gamma_p\left(\frac{1}{2}\right)}{\Gamma_p\left(\frac{3}{4}\right)}\pmod{p^3},
\quad\mathrm{if}\quad p\equiv 1\pmod4.
\end{align*}
Here and in what follows, $p$ is an odd prime  and
 the $p$-adic gamma function is defined as
 $$\Gamma_p(x)=\lim_{m\rightarrow x}\left(-1\right)^m\prod_{\substack{0< k< m\\ (k,p)=1}}k,$$
 where the limit is for $m$ tending to $x$ $p$-adically in $\Z_{\ge 0}$.

The Ramanujan-type supercongruences and their $q$-analogues have attracted a number of experts.
The reader interested in $q$-analogues of supercongruences is referred to
 \cite{Guo9, Guo_Schlosser2, Liu_Wang1, WangChen, Wang_Yu,  WeiChuanan, Xu_Wang}.
Especially,
a $q$-analogue of Van Hamme's (G.2) was proposed by Liu and Wang \cite{Liu_Wang2} as follows: for integers $n\equiv1\pmod4$,
\begin{equation}
\sum_{k=0}^{\frac{n-1}{4}}[8k+1]\frac{(q;q^4)^4_k}{(q^4;q^4)^4_k}q^{2k}
\equiv
\frac{(q^2;q^4)_{\frac{n-1}{4}}}{(q^4;q^4)_{\frac{n-1}{4}}}[n]q^{\frac{1-n}{4}} \pmod{[n]\Phi_n(q)^2}.
\end{equation}

We now need to introduce some necessary definitions and concepts.
Let $a$, $q$ be  complex numbers with $\left|q\right|<1$ and $n$ a nonnegative integer.
The {\em $q$-integer} is defined as $[n]=[n]_q=1+q+\cdots+q^{n-1}$.
And  the {\em $q$-shifted factorial} is defined  by
\begin{equation*}
(a;q)_0=1\quad \mathrm{and} \quad(a;q)_n=(1-a)(1-aq)\cdots(1-aq^{n-1}).
\end{equation*}
Usually,   the multiple $q$-shifted factorial is directly written  as
\begin{equation*}
(a_1,a_2,\ldots,a_r;q)_m=(a_1;q)_m (a_2;q)_m\cdots (a_r;q)_m,
\end{equation*}
where $r\in\N$ and $m\in\Z\cup\left\{\infty\right\}$.
Moreover,  the $n$-th {\em cyclotomic polynomial} in $q$ is represented by $\Phi_n(q)$:
\begin{align*}
\Phi_n(q)=\prod_{\substack{1\leqslant k\leqslant n\\ \gcd(n,k)=1}}\left(q-\zeta^k\right),
\end{align*}
where $\zeta$ is an $n$-th primitive root of unity.

By transforming $q$-supercongruences on double sums into the ones from squares of truncated
basic hypergeometric series,
El Bachraoui \cite{Bachraoui} obtained the following result: for any positive odd integer $n$ with $\gcd(n,6)=1$,
\begin{equation*}
\sum_{k=0}^{n-1}\sum_{j=0}^{k}c_q(j)c_q(k-j)
\equiv
q^{1-n}[n]^2 \pmod{[n]\Phi_n(q)^2},
\end{equation*}
where $c_q(k)=[8k+1](q;q^2)^2_k(q;q^2)_{2k}q^{2k^2}/((q^6;q^6)^2_k(q^2;q^2)_{2k})$.
Recently, many experts paid attention to El Bachraoui's work and obtained some new results (see, for example, \cite{Guo_Li, Song_Wang, LiLong}).

Inspired by the above work, in this paper,
by some essential tools including
the `creative microscoping' method introduced  by Guo and Zudilin \cite{Guo_Zudilin} and
 the Chinese remainder theorem for coprime polynomials, we shall
investigate several $q$-supercongruences on triple and quadruple sums.

\begin{theorem}\label{THM1}
Let $d$ and $n$ be positive integers with $d\geq3$ and $n\equiv1\pmod d$. For $k\geq 0$, let
\begin{equation*}
c_q(k)=[2dk+1]\frac{(q;q^d)^4_k}{(q^d;q^d)^4_k}q^{(d-2)k}.
\end{equation*}
Then, we have
\begin{align}\label{AE1}
\sum_{{i+j+k\leq n-1}}c_q(i)c_q(j)c_q(k)
\equiv
[n]^3q^{\frac{3(1-n)}{d}}\frac{(q^2;q^d)^3_{\frac{n-1}{d}}}{(q^d;q^d)^3_{\frac{n-1}{d}}}\pmod{[n]\Phi_n(q)^3}.
\end{align}
\end{theorem}


Letting $d=3$ in Theorem \ref{THM1}, we get the following result.
 \begin{corollary}
Let $n$ be a positive integer with $n\equiv1\pmod3$. For $k\geq0$, let
\begin{equation*}
c_q(k)=[6k+1]\frac{(q;q^3)^4_k}{(q^3;q^3)^4_k}q^{k}.
\end{equation*}
Then, modulo $[n]\Phi_n(q)^3$,
\begin{align*}
\sum_{{i+j+k\leq n-1}}c_q(i)c_q(j)c_q(k) \equiv
[n]^3q^{1-n}\frac{(q^2;q^3)^3_{\frac{n-1}{3}}}{(q^3;q^3)^3_{\frac{n-1}{3}}}.
\end{align*}
\end{corollary}

Note that He \cite{He} obtained the following Ramanujan-type
supercongruence: for primes $p\geq5$,
\begin{equation*}
\sum_{k=0}^{\frac{p-1}{3}}(6k+1)\frac{\left(\frac{1}{3}\right)^4_k}{k!^4}\equiv
-p\Gamma_p\left(\tfrac{1}{3}\right)^3\pmod {p^4}, \quad\mathrm{if}\
p\equiv1\pmod3,
\end{equation*}
and Liu and Wang \cite{Liu_Wang} presented several different
$q$-analogues of He's supercongruence.

Similarly, setting $d=4$ in Theorem \ref{THM1} yields a new $q$-supercongruence related to Van Hamme's (G.2) supercongruence.
\begin{corollary}
Let $n$ be a positive integer with $n\equiv1\pmod4$. For $k\geq0$, let
\begin{equation*}
c_q(k)=[8k+1]\frac{(q;q^4)^4_k}{(q^4;q^4)^4_k}q^{2k}.
\end{equation*}
Then, modulo $[n]\Phi_n(q)^3$,
\begin{align}
\sum_{{i+j+k\leq n-1}}c_q(i)c_q(j)c_q(k)
\equiv
[n]^3q^{\frac{3(1-n)}{4}}\frac{(q^2;q^4)^3_{\frac{n-1}{4}}}{(q^4;q^4)^3_{\frac{n-1}{4}}}
.\label{Equation1}
\end{align}
\end{corollary}

Moreover, we shall give a $q$-analogue of the quadruple sum  of (G.2) as follows.
\begin{theorem}\label{THM5}
Let $n$ be a positive integer with $n\equiv1\pmod4$. For $k\geq 0$, let
\begin{equation*}
c_q(k)=[8k+1]\frac{(q;q^4)^4_k}{(q^4;q^4)^4_k}q^{2k}.
\end{equation*}
Modulo $[n]\Phi_n(q)^4$, we have
\begin{align}
\sum_{{j_1+j_2+j_3+j_4\leq n-1}}c_q(j_1)c_q(j_2)c_q(j_3)c_q(j_4)
\equiv
[n]^8q^{4(1-n)}\left(
\sum_{k=0}^{\frac{n-1}{4}}\frac{(q,q,q^3;q^4)_k}{(q^4,q^4,q^5;q^4)_k}q^{4k}
\right)^4.
\end{align}\label{AAE4}
\end{theorem}

The rest of the paper is organized as follows. We shall prove Theorem \ref{THM1} in the next section by making use of the `creative microscoping'
method and the Chinese remainder theorem for coprime polynomials.
Finally, in Section 3, we give a proof of
Theorem \ref{THM5} through establishing its parametric generalization.

\section{Proof of Theorem \ref{THM1}}
We first give an important result due to El Bachraoui \cite{Bachraoui1}, which plays an important role in our proof.
\begin{lemma}\label{Lemma6}
Let $d$, $n$, $t$ be positive integers with $d\geq t\geq 2$ and $n\equiv 1\pmod d$.
Let $\{c(k)\}^\infty_{k=0}$ be a sequence of complex numbers.
If $c(k)=0$ for $(n-1)/d<k<n$, then
\begin{equation}\label{Q8}
\sum_{{k_1+\cdots+k_t\leq n-1 }}c(k_1)\cdots c(k_t)=\left(\sum_{k=0}^{\frac{n-1}{d}}c(k)\right)^t,
\end{equation}
Furthermore, if $c(ln+k)/c(ln)=c(k)$ for all nonnegative integers $k$ and $l$ such that $0\leq k\leq n-1$,
then,
\begin{align}
&\sum_{{k_1+\cdots+k_t=ln+k}}c(k_1)c(k_2)\cdots c(k_t)\nonumber\\
&\quad\quad=\sum_{{k_1+\cdots+k_t=k}}c(k_1)c(k_2)\cdots c(k_t)\nonumber\\
&\quad\quad\times\sum_{s_1=0}^{l}c(s_1n)\sum_{s_2=0}^{l-s_1}c(s_2n)\cdots
\sum_{s_{t-1}=0}^{l-s_1-\cdots-s_{t-2}}c(s_{t-1}n)c\left((l-s_1-\cdots-s_{t-1})n\right).\label{Q7}
\end{align}
\end{lemma}

In order to prove Theorem \ref{THM1}, we need to confirm the following two lemmas based on the result presented in \cite{Guo_Schlosser3}.
\begin{lemma}\label{Lem5}
Let $n$, $d$ be  positive integers  with $d\geq3$ and $n\equiv1\pmod d$, $a$ and $b$  indeterminates.  For $k\geq0$, let
\begin{equation*}
z_q(k)=[2dk+1]\frac{\left(q,aq,q/a,q/b;q^d\right)_k}{\left(q^d,aq^d,q^d/a,bq^d;q^d\right)_k}b^kq^{(d-2)k}.
\end{equation*}
Then,  we have, modulo $[n](1-aq^n)(a-q^n)$,
\begin{align}
\sum_{{i+j+k\leq n-1}}z_q(i)z_q(j)z_q(k)
\equiv
[n]^3\frac{\left(b/q\right)^{\frac{3(n-1)}{d}}\left(q^2/b;q^d\right)^3_{\frac{n-1}{d}}}{\left(bq^d;q^d\right)^3_{\frac{n-1}{d}}}
.\label{L6}
\end{align}
\end{lemma}
\begin{proof}
For $n=1$, the result is clearly true. We now assume that $n$ is an  integer with $n>1$.
Recall that Guo and Schlosser \cite[Equation (4.3)]{Guo_Schlosser3} gave the following result:
modulo ${[n](1-aq^n)(a-q^n)(b-q^n)}$,
\begin{align}
\sum_{k=0}^{M}z_q(k)
&\equiv[n]
\frac{\left(b/q\right)^{\frac{n-1}{d}}\left(q^{2}/b;q^d\right)_{\frac{n-1}{d}}}
{\left(bq^d;q^d\right)_{\frac{n-1}{d}}}\frac{(b-q^n)(ab-1-a^2+aq^n)}{(a-b)(1-ab)}\nonumber\\
&\quad+[n]\frac{(q,q^{d-1};q^d)_{\frac{n-1}{d}}}{(aq^d,q^d/a;q^d)_{\frac{n-1}{d}}}\frac{(1-aq^n)(a-q^n)}{(a-b)(1-ab)}
,\label{Equation3}
\end{align}
where $M=(n-1)/d$ or $n-1$.
It is clear to see that $\left(q;q^d\right)_k\equiv0\pmod{\Phi_n(q)}$ for $(n-1)/d<k<n$. This means that
$z_q(k)\equiv0 \pmod{\Phi_n(q)}$ for $(n-1)/d<k<n$. Therefore, employing  Lemma \ref{Lemma6} with $t=3$ produces
\begin{equation}\label{Equation4}
\sum_{{i+j+k\leq n-1}}z_q(i)z_q(j)z_q(k)
\equiv0
\pmod{\Phi_n(q)}.
\end{equation}
Let $\zeta\neq1$ be an $n$-th root of unity, not necessarily primitive.
Namely, $\zeta$ is a primitive $m$-th root of unity with $m\mid n$ and $m>1$.
From \eqref{Equation4}, we get
\begin{equation}\label{EEEE5}
\sum_{{i+j+k\leq m-1}}z_{\zeta}(i)z_{\zeta}(j)z_{\zeta}(k)=0.
\end{equation}
It is not difficult to check that $z_{\zeta}(lm+s)/z_{\zeta}(lm)=z_{\zeta}(s)$ for nonnegative integers $l$ and $s$ with $0\leq s\leq m-1$.
Applying  Lemma \ref{Lemma6} with $t=3$ again, we have
\begin{align}
&\sum_{{i+j+k\leq n-1}}z_{\zeta}(i)z_{\zeta}(j)z_{\zeta}(k)\nonumber\\
&\quad=\sum_{r=0}^{n-1}\sum_{i+j+k=r}z_{\zeta}(i)z_{\zeta}(j)z_{\zeta}(k)\nonumber\\
&\quad=\sum_{l=0}^{\frac{n}{m}-1}\sum_{s=0}^{m-1}\sum_{i+j+k=lm+s}z_{\zeta}(i)z_{\zeta}(j)z_{\zeta}(k)\nonumber\\
&\quad=\sum_{l=0}^{\frac{n}{m}-1}\sum_{a=0}^{l}z_{\zeta}(am)\sum_{t=0}^{l-a}z_{\zeta}(tm)z_{\zeta}((l-a-t)m)
\sum_{s=0}^{m-1}\sum_{i+j+k=s}z_{\zeta}(i)
z_{\zeta}(j)z_{\zeta}(k)\nonumber\\
&\quad=0.\nonumber
\end{align}
Since the above equality is true for any $n$-th root of unity $\zeta\neq1$, we conclude that
\begin{equation}\label{EEEE6}
\sum_{{i+j+k\leq n-1}}z_{q}(i)z_{q}(j)z_{q}(k)\equiv0\pmod{[n]}.
\end{equation}
On the other hand, letting $a=q^n$ or $q^{-n}$ in \eqref{Equation3},  we have
\begin{equation*}
\sum_{k=0}^{M}\tilde{z}_q(k)=[n]
\frac{\left(b/q\right)^{\frac{n-1}{d}}\left(q^{2}/b;q^d\right)_{\frac{n-1}{d}}}{\left(bq^d;q^d\right)_{\frac{n-1}{d}}}
,
\end{equation*}
where
\begin{equation*}
\tilde{z}_q(k)=[2dk+1]\frac{\left(q,q^{1+n},q^{1-n},q/b;q^d\right)_k}{\left(q^d,q^{d+n},q^{d-n},bq^d;q^d\right)_k}b^kq^{(d-2)k}.
\end{equation*}
Note the fact that $\tilde{z}_q(k)=0$ for $(n-1)/d<k<n$. Using Lemma \ref{Lemma6} with $t=3$ again, we arrive at
\begin{equation*}
\sum_{{i+j+k\leq n-1}}\tilde{z}_q(i)\tilde{z}_q(j)\tilde{z}_q(k)
=
[n]^3\frac{\left(b/q\right)^{\frac{3(n-1)}{d}}\left(q^2/b;q^d\right)^3_{\frac{n-1}{d}}}{\left(bq^d;q^d\right)^3_{\frac{n-1}{d}}}.
\end{equation*}
Then, we have the following $q$-congruence involving a variable $a$: modulo $\left(1-aq^n\right)\left(a-q^n\right)$,
\begin{equation}\label{Equation6}
\sum_{{i+j+k\leq n-1}}z_q(i)z_q(j)z_q(k)
\equiv
[n]^3\frac{\left(b/q\right)^{\frac{3(n-1)}{d}}\left(q^2/b;q^d\right)^3_{\frac{n-1}{d}}}{\left(bq^d;q^d\right)^3_{\frac{n-1}{d}}}.
\end{equation}
From \eqref{EEEE6} and \eqref{Equation6},
we conclude that \eqref{Equation6} is true modulo $[n]\left(1-aq^n\right)\left(a-q^n\right)$,
 as desired.
\end{proof}
\begin{lemma}\label{Lem6}
Let $n$, $d$ be  positive integers  with $d\geq3$ and $n\equiv1\pmod d$, $a$ and $b$ indeterminates. For $k\geq0$, let
\begin{equation*}
z_q(k)=[2dk+1]\frac{\left(q,aq,q/a,q/b;q^d\right)_k}{\left(q^d,aq^d,q^d/a,bq^d;q^d\right)_k}b^kq^{(d-2)k}.
\end{equation*}
Then, we have
\begin{align}
\sum_{{i+j+k\leq n-1}}z_q(i)z_q(j)z_q(k)
\equiv
[n]^3\frac{\left(q,q^{d-1};q^d\right)^3_{\frac{n-1}{d}}}{\left(aq^d,q^d/a;q^d\right)^3_{\frac{n-1}{d}}} \pmod{b-q^n}.
\label{L4}
\end{align}
\end{lemma}
\begin{proof}
Letting $b=q^n$ in \eqref{Equation3}, we get
\begin{equation*}
\sum_{k=0}^{M}\hat{z}_q(k)
=
[n]
\frac{\left(q,q^{d-1};q^d\right)_{\frac{n-1}{d}}}{\left(aq^d,q^d/a;q^d\right)_{\frac{n-1}{d}}},
\end{equation*}
where
\begin{equation*}
\hat{z}_q(k)=[2dk+1]\frac{\left(q,aq,q/a,q^{1-n};q^d\right)_k}{\left(q^d,aq^d,q^d/a,q^{d+n};q^d\right)_k}q^{(n+d-2)k}.
\end{equation*}
Note that $\hat{z}_q(k)=0$ for $(n-1)/d<k<n$. From Lemma \ref{Lemma6} with $t=3$, we obtain
\begin{equation*}
\sum_{{i+j+k\leq n-1}}\hat{z}_q(i)\hat{z}_q(j)\hat{z}_q(k)
=
[n]^3
\frac{\left(q,q^{d-1};q^d\right)^3_{\frac{n-1}{d}}}{\left(aq^d,q^d/a;q^d\right)^3_{\frac{n-1}{d}}},
\end{equation*}
which means that
\begin{equation*}
\sum_{{i+j+k\leq n-1}}z_q(i)z_q(j)z_q(k)
\equiv
[n]^3
\frac{\left(q,q^{d-1};q^d\right)^3_{\frac{n-1}{d}}}{\left(aq^d,q^d/a;q^d\right)^3_{\frac{n-1}{d}}}\pmod{b-q^n}
\end{equation*}
is true.
\end{proof}

\begin{proof}[Proof of Theorem \ref{THM1}]
It is easy to see that $[n]\left(1-aq^n\right)\left(a-q^n\right)$ and $b-q^n$ are relatively prime polynomials. Noting the relations
\begin{align*}
\frac{\left(b-q^n\right)(ab-1-a^2+aq^n)}{(a-b)(1-ab)}&\equiv 1 \pmod{\left(1-aq^n\right)(a-q^n)},\\
\frac{(1-aq^n)(a-q^n)}{(a-b)(1-ab)}&\equiv 1 \pmod{b-q^n},
\end{align*}
and employing the Chinese remainder theorem for coprime polynomials, we  derive the following result
from \eqref{L6} and \eqref{L4}: modulo $[n](1-aq^n)(a-q^n)(b-q^n)$,
\begin{align}
&\sum_{{i+j+k\leq n-1}}z_q(i)z_q(j)z_q(k)\nonumber\\
&\quad\quad\equiv
[n]^3\left\{
\frac{(b-q^n)(ab-1-a^2+aq^n)}{(a-b)(1-ab)}
\frac{\left(b/q \right)^{\frac{3(n-1)}{d}}\left(q^2/b;q^d\right)^3_{\frac{n-1}{d}}}{\left(bq^d;q^d\right)^3_{\frac{n-1}{d}}}\right.\nonumber\\
&\quad\quad\quad\quad\left.+\frac{(1-aq^n)(a-q^n)}{(a-b)(1-ab)}
\frac{\left(q,q^{d-1};q^d\right)^3_{\frac{n-1}{d}}}{\left(aq^d,q^d/a;q^d\right)^3_{\frac{n-1}{d}}}\right\}
.\label{L5}
\end{align}
Taking $b\rightarrow1$ in \eqref{L5}, we deduce that, modulo $\Phi_n(q)^2(1-aq^n)(a-q^n)$,
\begin{align}
\sum_{{i+j+k\leq n-1}}z_q(i)z_q(j)z_q(k)
\equiv
[n]^3
q^{\frac{3(1-n)}{d}}
\frac{(q^2;q^d)^3_{\frac{n-1}{d}}}{(q^d;q^d)^3_{\frac{n-1}{d}}}
.\label{Eequation7}
\end{align}
Here we have employed the relation:
\begin{equation*}
(1-q^n)(1+a^2-a-aq^n)=(1-a)^2+(1-aq^n)(a-q^n).
\end{equation*}
Finally, letting $a\rightarrow1$ in \eqref{Eequation7} and noting $1-q^n$ has the factor $\Phi_n(q)$,
we see that the  congruence \eqref{AE1} is true modulo $\Phi_n(q)^4$. On the other hand, our proof of
\eqref{EEEE6} is still valid for $a=b=1$, which means that \eqref{AE1} is also correct modulo $[n]$.
Since the least common multiple of $\Phi_n(q)^4$ and $[n]$ is $[n]\Phi_n(q)^3$, we finish the proof of Theorem \ref{THM1}.
\end{proof}

\section{Proof  of Theorem  \ref{THM5}}\label{Section4}
We first establish a parametric extension of Theorem \ref{THM5} as follows.
\begin{theorem}\label{THM3}
Let $d$ and $n$ be positive integers with $d\geq4$ and $n\equiv 1\pmod d$, $c$ an indeterminate. For $k\geq0$, let
\begin{equation*}
c_q(k)=[2dk+1]\frac{(q;q^d)^5_k(cq;q^d)_k}{(q^d;q^d)^5_k\left(q^d/c;q^d\right)_k}\left(\frac{q^{2d-3}}{c}\right)^k.
\end{equation*}
Then,  we have,  modulo $[n]\Phi_n(q)^4$,
\begin{align}
&\sum_{{j_1+j_2+j_3+j_4\leq n-1}}c_q(j_1)c_q(j_2)c_q(j_3)c_q(j_4)\nonumber\\
&\quad\quad\equiv
[n]^4\left(cq\right)^{\frac{4(1-n)}{d}}
\frac{\left(cq^2;q^d\right)^4_{\frac{n-1}{d}}}{(q^d/c;q^d)^4_{\frac{n-1}{d}}}
\left(\sum_{k=0}^{\frac{n-1}{d}}\frac{\left(q,q,cq,q^{d-1};q^d\right)_k}{\left(q^d,q^d,q^d,cq^2;q^d\right)_k}q^{dk}\right)^4
.
\label{Q1}
\end{align}
\end{theorem}
Clearly, Theorem \ref{THM5} can be gotten by setting $c=q^{d-1}$ and $d=4$ in Theorem \ref{THM3}.

In order to confirm Theorem \ref{THM3}, we need to prove the following parametric extension.
\begin{lemma}\label{THM4}
Let $d$ and $n$ be positive integers with $d\geq4$ and $n\equiv 1\pmod d$, $a$, $b$ and $c$  indeterminates. For $k\geq0$, let
\begin{equation*}
z_q(k,a,b)=[2dk+1]\frac{\left(aq,q/a,bq,q/b,cq,q;q^d\right)_k}{\left(q^d/q,aq^d,q^d/b,bq^d,q^d/c,q^d;q^d\right)_k}\left(\frac{q^{2d-3}}{c}\right)^k.
\end{equation*}
Then, modulo $[n](1-aq^n)(a-q^n)(1-bq^n)(b-q^n)$,
\begin{align}
&\sum_{{j_1+j_2+j_3+j_4\leq n-1}}z_q(j_1,a,b)z_q(j_2,a,b)z_q(j_3,a,b)z_q(j_4,a,b)\nonumber\\
&\quad\quad\equiv[n]^4\left(cq\right)^{\frac{4(1-n)}{d}}\frac{\left(cq^2;q^d\right)^4_{\frac{n-1}{d}}}{\left(q^d/c;q^d\right)_{\frac{n-1}{d}}}\nonumber\\
&\quad\quad \times
\left\{
\frac{(1-bq^n)(b-q^n)(-1-a^2+aq^n)}{(a-b)(1-ab)}
\left(\sum_{k=0}^{\frac{n-1}{d}}\frac{\left(aq,q/a,cq,q^{d-1};q^d\right)_k}{\left(bq^d,q^d/b,cq^2,q^d;q^d\right)_k}q^{dk}\right)^4\right.\nonumber\\
&\quad\quad \left.\quad+
\frac{(1-aq^n)(a-q^n)(-1-b^2+bq^n)}{(b-a)(1-ab)}
\left(\sum_{k=0}^{\frac{n-1}{d}}\frac{\left(bq,q/b,cq,q^{d-1};q^d\right)_k}{\left(aq^d,q^d/a,cq^2,q^d;q^d\right)_k}q^{dk}\right)^4
\right\}
.\label{AE4}
\end{align}
\end{lemma}
\begin{proof}
For $n=1$, the result is obviously  true. We now assume that $n$ is an integer with $n>1$.
Setting $r=1$ in \cite[Theorem 5.7]{WeiChuanan2}, we have, modulo $[n](1-aq^n)(a-q^n)(1-bq^n)(b-q^n)$,
\begin{align}
&\sum_{k=0}^{M}z_q(k,a,b)\nonumber\\
&\:\equiv
[n](cq)^{\frac{1-n}{d}}\frac{(cq^2;q^d)_{\frac{n-1}{d}}}{(q^d/c;q^d)_{\frac{n-1}{d}}}\nonumber\\
&\quad\times
\left\{
\frac{(1-bq^n)(b-q^n)(-1-a^2+aq^n)}{(a-b)(1-ab)}\sum_{k=0}^{\frac{n-1}{d}}
\frac{\left(aq,q/a,cq,q^{d-1};q^d\right)_k}{\left(aq^d,q^d/a,cq^2,q^d;q^d\right)_k}q^{dk}\right.\nonumber\\
&\quad\left.\quad+
\frac{(1-aq^n)(a-q^n)(-1-b^2+bq^n)}{(b-a)(1-ab)}\sum_{k=0}^{\frac{n-1}{d}}
\frac{\left(bq,q/b,cq,q^{d-1};q^d\right)_k}{\left(aq^d,q^d/a,cq^2,q^d;q^d\right)_k}q^{dk}
\right\},\label{AE9}
\end{align}
where $M=(n-1)/d$ or $n-1$.
Then, from \eqref{AE9}, we can deduce that
\begin{align}
\sum_{k=0}^{M}z_q(k,a,b)\equiv0\pmod{[n]}.\label{AE8}
\end{align}
Since $\left(q;q^d\right)_k\equiv 0\pmod{\Phi_n(q)}$ for $(n-1)/d<k<n$, we deduce that $z_q(k,a,b)\equiv0\pmod{\Phi_n(q)}$ for $(n-1)/d<k<n$.
Then, applying  Lemma \ref{Lemma6} with $t=4$, we arrive at
\begin{equation}\label{AEE2}
\sum_{{j_1+j_2+j_3+j_4\leq n-1}}z_q(j_1,a,b)z_q(j_2,a,b)z_q(j_3,a,b)z_q(j_4,a,b)
\equiv0\pmod{\Phi_n(q)}.
\end{equation}
Letting $\zeta\neq1$ be an $n$-th root of unity, not necessarily primitive.
Namely, $\zeta$ is a primitive $m$-th root of unity with $m\mid n$ and $m>1$.
From \eqref{AEE2}, we have
 \begin{equation}\label{AEE3}
\sum_{{j_1+j_2+j_3+j_4\leq m-1}}z_\zeta(j_1,a,b)z_\zeta(j_2,a,b)z_\zeta(j_3,a,b)z_\zeta(j_4,a,b)
=0.
\end{equation}
Observe that $z_\zeta(lm+k)/z_\zeta(lm)=z_\zeta(k)$ for nonnegative integers $l$ and $k$ with $0\leq k\leq m-1$.
Using Lemma \ref{Lemma6} with $t=4$ again, we obtain
\begin{align*}
&\sum_{{j_1+j_2+j_3+j_4\leq n-1}}z_\zeta(j_1,a,b)
z_\zeta(j_2,a,b)z_\zeta(j_3,a,b)z_\zeta(j_4,a,b)\\
&\quad=\sum_{r=0}^{n-1}
\sum_{{j_1+j_2+j_3+j_4= r}}z_\zeta(j_1,a,b)z_\zeta(j_2,a,b)z_\zeta(j_3,a,b)z_\zeta(j_4,a,b)\\
&\quad=\sum_{l=0}^{\frac{n}{m}-1}\sum_{k=0}^{m-1}\sum_{{j_1+j_2+j_3+j_4=lm+k}}
z_\zeta(j_1,a,b)z_\zeta(j_2,a,b)z_\zeta(j_3,a,b)z_\zeta(j_4,a,b)\\
&\quad=\sum_{l=0}^{\frac{n}{m}-1}\sum_{k_1=0}^{l}z_\zeta(k_1m,a,b)
\sum_{k_2=0}^{l-k_1}z_\zeta(k_2m,a,b)
\sum_{k_3=0}^{l-k_1-k_2}z_\zeta(k_3m,a,b)z_\zeta\left((l-k_1-k_2-k_3)m,a,b\right)\\
&\quad\quad\times\sum_{k=0}^{m-1}\sum_{{j_1+j_2+j_3+j_4=k}}z_\zeta(j_1,a,b)z_\zeta(j_2,a,b)z_\zeta(j_3,a,b)z_\zeta(j_4,a,b)\\
&\quad=0,
\end{align*}
which indicates that
\begin{align}\label{L1}
\sum_{\substack{j_1+j_2+j_3+j_4\leq n-1}}z_q(j_1,a,b)z_q(j_2,a,b)z_q(j_3,a,b)z_q(j_4,a,b)\equiv0\pmod{[n]}
\end{align}
is true.

Moreover, letting $a=q^n$ or $a=q^{-n}$ in \eqref{AE9} gives
\begin{align}\label{AAE1}
\sum_{k=0}^{M}z_q(k,q^n,b)
&=\sum_{k=0}^{M}z_q(k,q^{-n},b)\nonumber\\
&=
[n](cq)^{\frac{1-n}{d}}\frac{(cq^2;q^d)_{\frac{n-1}{d}}}{(q^d/c;q^d)_{\frac{n-1}{d}}}
\sum_{k=0}^{\frac{n-1}{d}}\frac{(q^{1-n},q^{1+n},cq,q^{d-1};q^d)_k}{(bq^d,q^d/b,cq^2,q^d;q^d)_k}q^{dk}.
\end{align}
Noticing that $z_q(k,q^n,b)=z_q(k,q^{-n},b)=0$ for $(n-1)/d<k<n$. Then, employing Lemma \ref{Lemma6} with $t=4$, we obtain, modulo $(1-aq^n)(a-q^n)$,
\begin{align}
&\sum_{{j_1+j_2+j_3+j_4\leq n-1}}z_q(j_1,a,b)z_q(j_2,a,b)z_q(j_3,a,b)z_q(j_4,a,b)\nonumber\\
&\quad\quad\quad\equiv[n]^4(cq)^{\frac{4(1-n)}{d}}\frac{(cq^2;q^d)^4_{\frac{n-1}{d}}}{(q^d/c;q^d)^4_{\frac{n-1}{d}}}
\left(\sum_{k=0}^{\frac{n-1}{d}}\frac{(aq,q/a,cq,q^{d-1};q^d)_k}{(bq^d,q^d/b,cq^2,q^d;q^d)_k}q^{dk}\right)^4.\label{AAE2}
\end{align}

Similarly, setting $b=q^n$ or $b=q^{-n}$ in \eqref{AE9} and utilizing Lemma \ref{Lemma6} with $t=4$ again, we get, modulo $(1-bq^n)(b-q^n)$,
\begin{align}
&\sum_{{j_1+j_2+j_3+j_4\leq n-1}}z_q(j_1,a,b)z_q(j_2,a,b)z_q(j_3,a,b)z_q(j_4,a,b)\nonumber\\
&\quad\quad\quad\equiv[n]^4(cq)^{\frac{4(1-n)}{d}}\frac{(cq^2;q^d)^4_{\frac{n-1}{d}}}{(q^d/c;q^d)^4_{\frac{n-1}{d}}}
\left(\sum_{k=0}^{\frac{n-1}{d}}\frac{(bq,q/b,cq,q^{d-1};q^d)_k}{(aq^d,q^d/a,cq^2,q^d;q^d)_k}q^{dk}\right)^4.\label{AAE3}
\end{align}

Noting that the polynomials $[n]$, $(1-aq^n)(a-q^n)$ and $(1-bq^n)(b-q^n)$ are relatively prime.
Then,  utilizing the Chinese remainder theorem for coprime polynomials and the relations:
\begin{align*}
\frac{(1-bq^n)(b-q^n)(-1-a^2+aq^n)}{(a-b)(1-ab)}&\equiv1\pmod{(1-aq^n)(a-q^n)},\\
\frac{(1-aq^n)(a-q^n)(-1-b^2+bq^n)}{(b-a)(1-ab)}&\equiv1\pmod{(1-bq^n)(b-q^n)},
\end{align*}
we can obtain the desired
result from \eqref{L1}, \eqref{AAE2} and \eqref{AAE3}.
\end{proof}

Subsequently, we start to prove Theorem \ref{THM3}.
\begin{proof}[Proof of Theorem \ref{THM3}]
Letting $b\rightarrow1$ in the congruence \eqref{AE4} produces, modulo $\Phi_n(q)^3(1-aq^n)(a-q^n)$,
\begin{align}\label{L-3}
&\sum_{{j_1+j_2+j_3+j_4\leq n-1}}z_q(j_1,a,1)z_q(j_2,a,1)z_q(j_3,a,1)z_q(j_4,a,1)\nonumber\\
&\quad\quad\equiv
[n]^4\left(cq\right)^{\frac{4(1-n)}{d}}\frac{(cq^2;q^d)^4_{\frac{n-1}{d}}}{(q^d/c;q^d)^4_{\frac{n-1}{d}}}\nonumber\\
&\quad\quad
\times
\left\{
\frac{(1-q^n)^2(-1-a^2+aq^n)}{(a-1)(1-a)}
\left(\sum_{k=0}^{\frac{n-1}{d}}\frac{(aq,q/a,cq,q^{d-1};q^d)_k}{(q^d,q^d,cq^2,q^d;q^d)_k}q^{dk}\right)^4\right.\nonumber\\
&\quad\left.\quad-
\frac{(1-aq^n)(a-q^n)(-2+q^n)}{(a-1)(1-a)}
\left(\sum_{k=0}^{\frac{n-1}{d}}\frac{(q,q,cq,q^{d-1};q^d)_k}{(aq^d,q^d/a,cq^2,q^d;q^d)_k}q^{dk}\right)^4
\right\}.
\end{align}
Next, setting $a\rightarrow1$ and utilizing the L'Hospital rule in \eqref{L-3}, we arrive at
\begin{align}
&\sum_{{j_1+j_2+j_3+j_4\leq n-1}}c_q(j_1)c_q(j_2)c_q(j_3)c_q(j_4)\nonumber\\
&\quad\equiv
[n]^4\left(cq\right)^{\frac{4(1-n)}{d}}
\frac{\left(cq^2;q^d\right)^4_{\frac{n-1}{d}}}{(q^d/c;q^d)^4_{\frac{n-1}{d}}}
\left(\sum_{k=0}^{\frac{n-1}{d}}\frac{\left(q,q,cq,q^{d-1};q^d\right)_k}{\left(q^d,q^d,q^d,cq^2;q^d\right)_k}q^{dk}\right)^4
\pmod{\Phi_n(q)^5}.
\label{AE6}
\end{align}
On the other hand, our proof of
\eqref{L1} is still valid for $a=b=1$, which means that \eqref{AE6} is also correct modulo $[n]$.
Since the least common multiple  of $\Phi_n(q)^5$ and $[n]$ is $[n]\Phi_n(q)^4$, we derive  Theorem \ref{THM3}
\end{proof}


\begin{thebibliography}{99}
\bibitem{Bachraoui}M. El Bachraoui,
On supercongruences for truncated sums of squares of basic hypergeometric series,
 Ramanujan J. \textbf{54} (2021), 415--426.


\bibitem{Bachraoui1}M. El Bachraoui,
$N$-tuple sum analogues for Ramanujan-type congruences,
arXiv: 2112. 00308.

\bibitem{Berndt}
B.C. Berndt and R.A. Rankin, Ramanujan: letters and commentary, History of Mathematics \textbf{9},
 Amer. Math. Soc., Providence, RI; London Math. Soc., London (1995).
\bibitem{Guo9}V.J.W. Guo,
Further $q$-supercongruences from a transformation of Rahman, J. Math. Anal. Appl. \textbf{511} (2022), Art. 126062.

\bibitem{Guo_Li}V.J.W. Guo and L. Li,
$q$-Supercongruences from squares of basic hypergeometric series,
arXiv: 2112. 12076.

\bibitem{Guo_Schlosser2}V.J.W. Guo and M.J. Schlosser,
 A family of $q$-hypergeometric congruences modulo the fourth power of a cyclotomic polynomial,
 Israel J. Math. \textbf{240} (2020), 821--835.

\bibitem{Guo_Schlosser3}V.J.W. Guo and M.J. Schlosser,
A new family of $q$-supercongruences modulo the fourth power of a cyclotomic polynomial,
Results Math. \textbf{75} (2020), Art. 155.


\bibitem{Guo_Zudilin}V.J.W. Guo and W. Zudilin,
A $q$-microscope for supercongruences,
 Adv. Math.
 \textbf{346} (2019), 329--358.

\bibitem{He}B. He, Supercongruences on truncated hypergeometric Series, Results Math. \textbf{72} (2017), 303--317.

\bibitem{LiLong}L. Li,
Some $q$-supercongruences for truncated forms of squares of basic hypergeometric series,
  J. Difference Equ. Appl. \textbf{27} (2021), 16--25.


\bibitem{Liu_Wang2}Y. Liu and X. Wang, $q$-Analogues of the (G.2) supergcongruence of Van Hamme,
 Rocky Mountain J. Math. \textbf{51} (2021), 1329--1340.


\bibitem{Liu_Wang}Y. Liu and X. Wang, $q$-Analogues of two Ramanujan-type supercongrucences,
J. Math. Anal. Appl. \textbf{502} (2021), Art. 125238.


\bibitem{Liu_Wang1}
Y. Liu and X. Wang, Some $q$-supercongruences from a quadratic transformation by Rahman,
 Results Math. \textbf{77} (2022), Art. 44.



 \bibitem{Song_Wang}
 H. Song and C. Wang,
  Some $q$-supercongruences modulo the fifth power of a cyclotomic polynomial from squares of $q$-hypergeometric series,
  Results Math. \textbf{76} (2021), Art. 222.


\bibitem{Van}L. Van Hamme, Some conjectures concerning partial sums of generalized hypergeometric series,
in: $p$-Adic Functional Analysis (Nijmegen, 1996),
Lecture Notes in Pure and Appl. Math. \textbf{192}, Dekker, New York (1997), 223--236.

\bibitem{WangChen}C. Wang,
A new $q$-extension of the (H.2) congruence of Van Hamme for primes $p\equiv1\pmod4$,
Results Math. \textbf{76} (2021), Art. 205.





\bibitem{Wang_Yu}X. Wang and M. Yu,  Some generalizations of a congruence by Sun and Tauraso,
  Period. Math. Hungar. (2021).
 https://doi.org/10.1007/s10998--021--00432--8.

\bibitem{WeiChuanan}C. Wei, A further $q$-analogue of Van Hamme's (H.2) supercongruence for any prime $p\equiv1\pmod4$,
Results Math. \textbf{76} (2021), Art. 92.
\bibitem{WeiChuanan2}C. Wei, Some $q$-supercongruences modulo the fifth and sixth powers of a cyclotomic polynomial,
arXiv: 2104. 07025.
\bibitem{Xu_Wang}C. Xu and X. Wang,
Proofs of Guo and Schlosser's two conjectures,
Period. Math. Hungar. (2022).
https://doi.org/10.1007/s10998--022--00452--y.
\end{thebibliography}
\end{document}